\documentclass[12pt]{article}
\usepackage{mathrsfs}
\usepackage{amsmath}
\usepackage{amsthm}
\usepackage{bbm}
\usepackage{subfloat}
\usepackage{subfigure}
\usepackage{fancyhdr,graphicx}
\usepackage{amsfonts}
\usepackage{amssymb}
\usepackage{latexsym,bm}
\usepackage{newlfont}
\usepackage{float}

\newtheorem{thm}{Theorem}[section]
\newtheorem{Lemma}[thm]{Lemma}
\newtheorem{cor}[thm]{Corollary}

\newtheorem{Remark}[thm]{Remark}
\newtheorem{Algorithm}[thm]{Algorithm}

\makeatletter
\long\def\@makecaption#1#2{%
 \vskip\abovecaptionskip
  \sbox\@tempboxa{{#1.}\quad #2}%
 \ifdim \wd\@tempboxa >\hsize
    { #1.}\quad #2\par
     \else
  \global \@minipagefalse
   \hb@xt@\hsize{\hfil\box\@tempboxa\hfil}%
   \fi
   \vskip\belowcaptionskip}
\makeatother

\title{Computing the permanental polynomials of bipartite graphs by Pfaffian orientation\footnote{This work is supported by NSFC (grant no. 10831001).
       }}

\author{Heping Zhang\footnote{The corresponding author.}, Wei Li\\
\small{School of Mathematics and Statistics, Lanzhou
University, Lanzhou, Gansu 730000, P. R. China.}\\
\small{E-mail addresses: zhanghp@lzu.edu.cn, li$\_$w07@lzu.cn.
}}

\date{}

\begin{document}

\maketitle

\makeatletter
\newcommand{\rmnum}[1]{\romannumeral #1} 　　
\newcommand{\Rmnum}[1]{\expandafter\@slowromancap\romannumeral #1@}
\makeatother

\begin{abstract}
The permanental polynomial of a graph $G$ is
$\pi(G,x)\triangleq\mathrm{per}(xI-A(G))$.  From  the result that a
bipartite graph $G$ admits an orientation $G^e$ such that every
cycle is oddly oriented if and only if it contains no even
subdivision of $K_{2,3}$, Yan and Zhang showed that the permanental
polynomial of such a bipartite  graph $G$ can be expressed as the
characteristic polynomial of the  skew adjacency matrix $A(G^e)$. In
this paper we first prove that this equality holds only if the
bipartite graph $G$ contains no even subdivision of $K_{2,3}$. Then
we prove that such bipartite graphs are planar.  Further we mainly
show that a 2-connected  bipartite graph contains no even
subdivision of $K_{2,3}$ if and only if it is planar 1-cycle
resonant. This implies that  each cycle is oddly oriented in any
Pfaffian orientation of a 2-connected bipartite graph containing no
even subdivision of $K_{2,3}$. As applications, permanental
polynomials for some types of bipartite graphs are computed.


\end{abstract}
%%introduction--------------------------------------------------

\section{Introduction}
\setlength{\unitlength}{1cm}

In this article, we always consider finite and simple graphs. Let
$G$ be a graph with vertex-set $V(G)=\{v_{1}, v_{2},\cdots,v_{n}\} $
and edge-set $E(G)=\{e_{1}, e_{2},\cdots, e_{m}\}$. The adjacency
matrix $A(G)=(a_{ij})_{n\times n}$ of $G$ is defined as
\begin{displaymath}
a_{ij}=\left\{
\begin{array}{ll}
 1 & \textrm{if vertex $v_{i}$ is adjacent to vertex $v_{j}$},\\
 0 & \textrm{otherwise}.
 \end{array}\right.
\end{displaymath}

The permanent of an $n\times n$ matrix $A$ is defined as\\
\begin{equation}
\mathrm{per(A)}=\sum_ {\sigma\in\Lambda_n}\prod_{i=1}^{n}a_{i\sigma(i)},
\end{equation}
where $\Lambda_{n}$ denotes the set of all permutations of
$\{1, 2,\cdots, n\}$. The permanental polynomial of $G$ is
defined as
\begin{equation}
\pi(G,x)=\mathrm{per}(xI-A(G))=\sum_{k=0}^{n}b_{k}x^{n-k},
\end{equation}
where $I$ is the identity matrix of order $n$. It is easy to see
that $(-1)^{k}b_{k}$ is the sum of the $k \times k$ principle
subpermanents of $A$ \cite{R.}. It was  mentioned in \cite{M., R.}
that if $G$ is a bipartite graph,  then $b_{2k}\geq 0$ and $
b_{2k+1}=0$ for all $k\geq 0$; In particular,
\begin{equation}b_n=\textrm{per}(A(G))=m^2(G),\end{equation}
where $m(G)$ is the number of
perfect matchings of $G$.

It was in 1981 that the permanental polynomial was firstly
investigated in chemistry in $\cite{Kasum.}$,   where some relations
between the permanental polynomial and the structure of conjugated
molecules were discussed. Later, Cash \cite{Pe, Ca} investigated the
mathematical properties of the coefficients and zeros of the
permanental polynomials of some chemical graphs. This suggests that
the permanental polynomial encodes a
 variety of structural information. Gutman and Cash  also demonstrated $\cite{I.}$ several relations between the coefficients of
 the permanental and characteristic polynomials.
 See \cite{Ca2, B, Huo.} for more about the permanental polynomials.
 We can compute the determinant of
 an $n\times n$ matrix efficiently by Gaussian elimination.
 Although  $\mathrm{per}(A)$ looks similar to
  $\mathrm{det}(A)$,  it is harder to be computed. Valiant proved  \cite{L.} that it is a
  $\#$P-complete problem to evaluate the permanent of a $(0, 1)$-matrix.
  For these reasons, we want to compute the permanental polynomials of graphs by the determinant of a matrix.

Pfaffian orientations of graphs can be used to the enumeration of
perfect matchings. Now we recall some definitions. A subgraph $H$ of
a graph $G$ is called \textit{nice} if $G- V(H)$ has a perfect
matching. Let $G^{e}$ be an orientation of a graph $G$ and $C$ a
cycle of even length in $G$. $C$ is \textit{oddly oriented} in
$G^{e}$ if  $C$ contains an odd number of edges that are directed in
either direction of the cycle. An orientation of $G$ is
\textit{Pfaffian} if every nice cycle $C$  is oddly oriented. The
skew adjacency matrix of $G^{e}$, denoted by $A(G^{e})$, is defined
as follows,
\begin{displaymath}
a_{ij}=\left\{
\begin{array}{lll}
 1 & \textrm{if $(v_{i},v_{j}) \in E(G^{e})$},\\
 -1 & \textrm{if $(v_{j},v_{i}) \in E(G^{e})$},\\
 0 & \textrm{otherwise.}
 \end{array}\right.
\end{displaymath}
If a bipartite graph $G$ has a Pfaffian orientation $G^e$, then \cite{Lo}
\begin{equation}\textrm{per}(A(G))=\textrm{det}(A(G^e))=m^2(G).\end{equation}

We say an edge $e$ of a graph $G$ is \textit{oddly subdivided} in a
graph $G^{\prime}$ if $G^{\prime}$ is obtained from $G$ by
subdividing $e$ an odd number of times. Otherwise, we say $e$ is
evenly subdivided. The graph $G^{\prime}$ is said to be \textit{an even
subdivision of a graph $G$} if $G'$ can be obtained from $G$ by
subdividing evenly each edge of  $G$.

Fisher and Little \cite{Fi} gave a characterization of a graph that
has an orientation with each cycle being oddly oriented. From this
result, we have  that a bipartite graph $G$ admits an orientation
$G^e$ such that every cycle is oddly oriented if and only if it
contains no even subdivision of $K_{2,3}$. Accordingly,  Yan and
Zhang showed \cite{W.} that the permanental polynomial of such a
bipartite graph $G$ can be expressed as the characteristic
polynomial of the skew adjacency matrix $A(G^e)$:
$\pi(G,x)=\mathrm{det}(xI-A(G^{e}))$.

In this paper we want to  investigate bipartite graphs for which the
permanental polynomials can be computed by the characteristic
polynomials. Following Yan and Zhang's  result, in Section 2 we
obtain that a bipartite graph admits this property if and only if it
has a Pfaffian orientation such that all the cycles
 are oddly oriented, if and only if it contains no
even subdivision of $K_{2,3}$.

In Section 3 we mainly  characterize bipartite graphs containing no
even subdivision of $K_{2,3}$.  The starting point of this section
is to show that these graphs are planar. Then we find that if such
graphs are 2-connected, they each has a bipartite ear decomposition
starting with any cycle. This implies that such graphs are
elementary bipartite graphs. Unexpectedly,  we obtain the  main
result that a 2-connected bipartite graph $G$ contains no even
subdivision of $K_{2,3}$ if and only if it is planar 1-cycle
resonant. In fact,  planar 1-cycle resonant graphs have already been
introduced and investigated extensively in \cite{Guo2, Guo1, Guo,
Xu}. Various characterizations \cite{Guo1} and constructional
features \cite{Guo} for planar 1-cycle resonant graphs have been
given. These characterizations enable ones to design efficient
algorithms to decide whether a 2-connected planar bipartite graph is
1-cycle resonant \cite{Guo, Xu}.

 From the previous main result,  we have that each cycle of
 a 2-connected bipartite graph $G$ containing  no
even subdivision of $K_{2,3}$ is nice and is oddly oriented in any
Pfaffian orientation of $G$. Finally an algorithm \cite{Lo} is
recalled  to give a Pfaffian orientation of a plane graph. As
applications, the permanental polynomials of some types of bipartite
graphs are computed. In particular, we obtain explicit expressions
for permanental polynomials of two classes of graphs $G_{1}^{s}$ and
$G_{2}^{r}$. \noindent
%%%%%%%%%%%%%%%%%%%%%%%%%%%%%%%%%%%%%%%%%%%%%%%%%%%%%%%%%%%%%%%%%%%%%%%%%%%%%%%%%%%
%第二部分
\section{A criterion for computing permanental polynomials}
%A subgraph $H$ of a graph $G$ is called \textit{central} (or {\em
%nice} in \cite{Lo}) if $G- V(H)$ has a perfect matching.
An elegant characterization for Pfaffian bipartite graphs was given
by Little.

\begin{thm} \cite{Li} \label{Pfaffian}
  A bipartite graph admits a Pfaffian orientation if and only if it
does not contain an even subdivision of $K_{3,3}$ as a nice
subgraph.
\end{thm}

   Fisher and Little  gave a characterization for the
existence of an orientation of a graph such that all the even cycles
 are oddly oriented as follows.

\begin{thm}\cite{Fi}\label{little}
A graph has an orientation under which every cycle of even length is
oddly oriented if and only if the graph contains no subgraph which
is, after the contraction of at most one cycle of odd length, an
even subdivision of $K_{2,3}$.
\end{thm}

For bipartite graphs, we have the following immediate corollary.

\begin{cor} \label{bipartite}
There exists an orientation of  a bipartite graph $G$ such that all
the cycles of  $G$ are oddly oriented if and only if $G$ contains no
even subdivision of $K_{2,3}$.
\end{cor}

Based on such results, Yan and Zhang found that the permanental polynomial of a
bipartite graph that has no even subdivision of $K_{2,3}$ can be
computed by the determinant.

  \begin{thm}\cite{W.} \label{Yan }
Let $G$ be a bipartite graph containing no even subdivision of
$K_{2,3}$. Then there exists an orientation $G^{e}$ of $G$ such that
\begin{align}
\pi(G,x)=\mathrm{det}(xI-A(G^{e})).
\end{align}
%where $A(G^{e}))$ denotes the skew adjacency matrix of
%$G^{e}$.
\end{thm}

 In fact we can show that the converse of the theorem is also valid.

\begin{thm}\label{orientation}
Let $G$ be a bipartite graph of order $n$. Then an orientation
$G^{e}$ of $G$ satisfies  $\pi(G,x)=\mathrm{det}(xI-A(G^{e}))$ if
and only if each cycle of $G$ is oddly oriented in $G^e$.
\end{thm}
\begin{proof}  Let
$$\mathrm{det}(xI-A(G^{e}))=\sum_{i=0}^na_{i}x^{n-i}.$$
Then  we have
$$a_{2k+1}=b_{2k+1} =0, k=0,1,2,...,$$
and $$b_{2k}= \sum\limits_H\mathrm{per}(A(H)),
a_{2k}=\sum\limits_H\mathrm{det}(A(H^{e})),k=0,1,...,$$ where both
sums range over all induced subgraphs $H$ of $G$ with $2k$ vertices.

For the complete we reprove the sufficiency. If $G$ has an
orientation $G^e$ such that every cycle is oddly oriented, then the
restriction of $G^e$ on each induced subgraph $H$ of $G$ is a
Pfaffian orientation $H^e$. Hence
$\textrm{per}(A(H))=\textrm{det}(A(H^e))$. That means that
$b_{i}=a_{i}$ for each $i\geq 1$. Hence
$\pi(G,x)=\mathrm{det}(xI-A(G^{e}))$.

We now prove the necessity. Suppose that  $G$ has a cycle $C$ that
is evenly oriented in $G^{e}$. For convenience, let
$V(C)=\{1,2,\ldots,2k\}$ and $C=12\cdots(2k)1$.  Now we consider the
subgraph $G[C]$ induced by the vertices of $C$. Let $G^{e}[C]$ be
the orientation $G^{e}$ restricted on $G[C]$. From the definitions
of permanents and determinants, we have  that
$\mathrm{per}(A(G[C]))=\sum\limits_{\sigma}
a_{1\sigma(1)}a_{2\sigma(2)}\cdots a_{2k\sigma(2k)}$ and
$\mathrm{det}(A(G^{e}[C]))=\sum\limits_{\sigma}
\text{sgn}(\sigma)a^{\prime}_{1\sigma(1)}a^{\prime}_{2\sigma(2)}\cdots
a^{\prime}_{2k\sigma(2k)}$, where $a_{i\sigma(i)}$ and
$a^{\prime}_{i\sigma(i)}$ are the elements in the $i$-th row and
$\sigma(i)$-th column of $A(G[C])$ and $A(G^{e}[C])$, respectively,
and both  sums range over all the permutations of $\{1,
2,\cdots,2k\}$. For any given $\sigma$, if some $a_{i\sigma(i)}$
equals zero, then $a_{1\sigma(1)}a_{2\sigma(2)}\cdots
a_{2k\sigma(2k)}
=a^{\prime}_{1\sigma(1)}a^{\prime}_{2\sigma(2)}\cdots
a^{\prime}_{2k\sigma(2k)}=0$. Otherwise,
$a_{1\sigma(1)}a_{2\sigma(2)}\cdots a_{2k\sigma(2k)}=1$ and
$a^{\prime}_{1\sigma(1)}a^{\prime}_{2\sigma(2)}\cdots
a^{\prime}_{2k\sigma(2k)}=1\,$ or $\,-1$. Hence
$a_{1\sigma(1)}a_{2\sigma(2)}\cdots a_{2k\sigma(2k)} \geq
\text{sgn}(\sigma)
a^{\prime}_{1\sigma(1)}a^{\prime}_{2\sigma(2)}\cdots
a^{\prime}_{2k\sigma(2k)}$. For the cycle $\sigma=(1 2\cdots2k)$, we
know that $\text{sgn}(\sigma)=-1$ and
$a^{\prime}_{1\sigma(1)}a^{\prime} _{2\sigma(2)}\cdots
a^{\prime}_{2k\sigma(2k)}=1$ since $C$ is evenly oriented. So
$\text{sgn}(\sigma)a^{\prime}_{1\sigma(1)}a^{\prime}_{2\sigma(2)}\cdots
a^{\prime}_{2k\sigma(2k)}=-1$, while
$a_{1\sigma(1)}a_{2\sigma(2)}\cdots a_{2k\sigma(2k)}=1$. Hence
$\mathrm{det}(A(G^{e}[C]))<\mathrm{per}(A(G[C]))$.

In an analogous argument, we have that
$\mathrm{det}(A(G^{e}[H]))\leq\mathrm{per}(A(G[H]))$ for any induced
subgraph $H$ of $G$. Hence
$a_{2k}=\sum\limits_H\mathrm{det}(A(H^{e}))<\sum\limits_H\mathrm{per}(A(H))=b_{2k}$
and $\pi (G, x)\neq \mathrm{det}(xI-A(H^{e}))$.
\end{proof}

\begin{cor} \label{criterion}There exists an orientation $G^{e}$ of a bipartite graph $G$ such that
$\pi(G,x)=\mathrm{det}(xI-A(G^{e}))$  if and only if $G$  contains
no even subdivision of $K_{2,3}$.
\end{cor}

\begin{proof} It follows immediately from Theorem \ref{orientation} and
Corollary \ref{bipartite}.
 \end{proof}

%%%%%%%%%%%%%%%%%%%%%%%%%%%%%%%%%%%%%%%%%%%%%%%%%%%%%%%%%%%%%%%%%%%%%%%%%%%%%%%%%%%%
\section{Characterizations and recognition of bipartite graphs containing no even subdivision of $K_{2,3}$}
In this section we will characterize bipartite graphs containing no
even subdivision of $K_{2,3}$. We first show that such graphs are
planar. Then we show that a 2-connected bipartite graph  has no even
subdivision of $K_{2,3}$ if and only if it is a planar 1-cycle
resonant graph.

\begin{thm}\cite[Kuratowski's theorem]{J.} \label{planar}
A graph is planar if and only if it contains no subdivision of
either $K_{5}$ or $K_{3,3}$.
\end{thm}

%We say an edge $e$ of a graph $G$ is \textit{oddly (resp. evenly)
%subdivided} in a graph $G^{\prime}$ if $G^{\prime}$ is obtained from
 %$G$ by subdividing $e$ an odd (resp. even) number of times.

\begin{Lemma}\label{bipartite planar}
If a bipartite graph contains no even subdivision of $K_{2,3}$, then
it is planar.
\end{Lemma}
\begin{proof}
We prove the converse-negative proposition.   Suppose that  a
bipartite graph $G$ is not planar. Then by Theorem \ref{planar} it
contains a subdivision $H$ of  $K_{3,3}$ or $K_{5}$.

If an edge  $e$ of $K_{3,3}$ or $K_5$ is oddly subdivided in $H$,
then  $P$ is the path of even length in $H$  obtained by subdividing
$e$. There are two cycles $C_{1}$, $C_{2}$  containing $e$ in
$K_{3,3}$ or $K_5$ with $C_{1}\cap C_{2}=\{e\}$, and the
corresponding cycles in $H$ are denoted by $C_{1}^{\prime}$,
$C_{2}^{\prime}$. Since $H$ is bipartite,  in $C_{1}^{\prime}\cup
C_{2}^{\prime}$ there are three pairwise internally disjoint paths
of even length connecting the two endvertices of $P$. So we obtain
an even subdivision of $K_{2,3}$.

If  $H$ is an even subdivision of $K_{3,3}$ or $K_{5}$, then
$K_{3,3}$ or $K_{5}$ always contains  $K_{2,3}$ as a subgraph, which
corresponds to an even subdivision of $K_{2,3}$ in $H$.
%(If  $H$ is an even subdivision of $K_{3,3}$, then
%$K_{3,3}$ always contains  $K_{2,3}$ as a subgraph, which
%corresponds to an even subdivision of $K_{2,3}$ in $H$.
%$H$ cannot be an even subdivision of $K_{5}$. Otherwise, an
%odd cycle exists in $H$, contradicting that $H$ is bipartite.)
\end{proof}

%First we suppose that $G$ contains a subgraph $H_{1}$ which is
%obtained by subdividing $K_{3,3}$. Since $H_{1}$ is bipartite, any
%cycle $C$ in it must be obtained by subdividing a cycle $C_{0}$ of
%$K_{3,3}$ an even number of times. So we have that an even number of
%edges of $C$ are evenly subdivided. If  every edge of $K_{3,3}$ is
%evenly subdivided in $H_{1}$, then an even subdivision of $K_{2,3}$
%exists in $H_{1}$.

%Because $H_{2}$ is bipartite, an even number of edges of any cycle
%$C$ in $K_{5}$ are evenly subdivided in $H_{2}$. $H_{2}$ can not be
%obtained by evenly subdividing every edge of $K_{5}$. Otherwise, a
%cycle of odd length is obtained in $H_{2}$, contradicting to the
%condition that $H_{2}$ is bipartite. So $H_{2}$ must be obtained by
%oddly subdividing some edge $e$ of $K_{5}$, and let $P$ be the path
%in $H_{2}$ obtained by subdividing $e$. Let $C_{1}$, $C_{2}$ be the
%triangles in $K_{5}$ with $C_{1}\cap C_{2}=\{e\}$ and the the
%corresponding cycles in $H_{2}$ are denoted by $C_{1}^{\prime}$,
%$C_{2}^{\prime}$. Then in $C_{1}^{\prime}\cup C_{2}^{\prime}$ there
%are three pairwise internally disjoint paths of even length
%connecting the two endvertices of $P$. So an even subdivision of
%$K_{2,3}$ exists.

A sequence of subgraphs of $G$, $(G_{0}, G_{1},\cdots, G_{m})$ is a \textit{bipartite ear decomposition} of $G$
if $G_{0}$ is an edge $x$, $G_{m}=G$, every $G_{i}$ for
$i=1, 2,\cdots,m$ is obtained from $G_{i-1}$ by adding an path $P_{i-1}$ of odd length which is openly disjoint
from $G_{i-1}$ but its endvertices belong to $G_{i-1}$.
Such an ear decomposition can also be denoted as $G:=x+P_{0}+P_{1}+\cdots+P_{m-1}$.

\begin{Lemma}   \label{is bipartite ear}
The 2-connected bipartite graph $G$ containing no even subdivision
of $K_{2,3}$ has a bipartite ear decomposition starting with any
cycle in it.
\end{Lemma}

\begin{proof}
Let $G_{1}$ be any cycle in $G$. We want to show that $G$ has a
bipartite ear decomposition starting with $G_{1}$ by induction. Let
$G_{i}$ be a subgraph of $G$ obtained by successively adding $i-1$
ears, $i\geq1$. If $G_{i}\neq G$, then an edge $uv$ exists in
$G-E(G_{i})$. For  an edge $xy\in E(G_{i})$, $uv$ and $xy$ lie on a
common cycle $C$ since  $G$ is 2-connected. Let $P$ be a path in $C$
such that the intersections of $P$ and $G_{i}$ are the both
endvertices $a$ and $b$ of $P$. If $a$ and $b$ have the same color,
then $P$ is a path of even length. Since $G_i$ is 2-connected, $G_i$
has two internally disjoint paths of even length connecting $a$ and
$b$.  Hence three pairwise internally disjoint paths of even length of $G$
connect $a$ and $b$. That is,  an even subdivision of $K_{2,3}$
exists in $G$ and a contradiction occurs. So $a$ and $b$ have
different colors and $P$ is a path of odd length. Now
$G_{i+1}:=G_i+P$ is a subgraph of $G$ obtained by successively
adding $i$ ears, $i\geq1$.  Hence $G$ has a bipartite ear
decomposition starting with any cycle in it.
\end{proof}

The above result shows that such graphs relate with 1-cycle resonant
graphs.  A connected graph is said to be \textit{$k$-cycle resonant}
if, for $1\leq t\leq k$, any $t$ disjoint cycles in $G$ are mutually
resonant, that is, there is a perfect matching $M$ of $G$ such that
each of the $t$ cycles is an $M$-alternating cycle. $k$-cycle
resonant graphs were  introduced by  Guo and Zhang \cite{Guo2} as a
natural generalization of $k$-resonant benzenoid systems, which
originate from Clar's aromatic sextet theory and Randi\'{c}'s
conjugated circuit model. They obtained that a $k$-cycle resonant
graph is bipartite \cite{Guo2}. In the following theorem, we can see
that such two types of graphs are equivalent under the 2-connected
condition.

\begin{thm} \label{is 1-cycle resonant}
A 2-connected bipartite graph $G$ contains no even subdivision of
$K_{2,3}$  if and only if $G$ is planar 1-cycle resonant.
\end{thm}

\begin{proof}
 We first prove the necessity. From Lemma \ref{is bipartite ear},  $G$ has a bipartite ear decomposition $G:=C+P_{1}+P_{2}+\cdots+P_{r}$ starting with any cycle $C$ in it.
Because the ears ${P_{i}}^{\prime}$s for $1\leq i\leq r$ are all of
odd length, the the graph $G-V(C)$ has a perfect matching $M$ which
covers all the internal vertices of each $P_{i}$. So every cycle of
$G$ is nice.  By Lemma \ref{bipartite planar} $G$ is planar. Hence
$G$ is  planar 1-cycle resonant.

Now we prove the sufficiency.  If $G$ contains a subgraph $H$ which
is an even subdivision of $K_{2,3}$, then in the subgraph $H$, there
are three pairwise internally disjoint paths $l_{1}$, $l_{2}$ and
$l_{3}$ of even length joining two given vertices. Since $G$ is
planar, we imbed it in the plane so that  $l_{2}$ lies in the
interior of the cycle $C:=l_{1}\cup l_{3}$. Since $G$ is 1-cycle
resonant,  $C_{1}:=l_{1}\cup l_{2}$ and  $C_{2}:=l_{2}\cup l_{3}$
are nice cycles of $G$. Hence there is an even number of vertices in
the interior of $C_1$ and $C_2$ respectively. Further, since $l_2$
has an odd number of internal vertices, there is an odd number of
vertices in the interior of $C$. This implies that $C$ is not a nice
cycle of $G$, contradicting that $G$ is 1-cycle resonant.
\end{proof}

A connected bipartite graph $G$ is {\em elementary} if each edge is
contained in a perfect matching of $G$.
For more details about such graphs see \cite{Zhang}.
Then by Lemmas
\ref{bipartite planar} and \ref{is bipartite ear} or Theorem \ref{is
1-cycle resonant} we have the following result.

\begin{cor}The 2-connected bipartite graph $G$ containing no even subdivision
of $K_{2,3}$ is a planar and elementary bipartite graph.
\end{cor}

 A \textit{block} of a connected
graph $G$ is a maximal connected subgraph of $G$ without
cutvertices. From Theorem \ref{is 1-cycle resonant} we  have the
following general result.

\begin{cor}A connected bipartite graph $G$ contains no even subdivision of $K_{2,3}$
if and only if each block of $G$ is planar 1-cycle resonant.
\end{cor}

From the proof of Theorem \ref{is 1-cycle resonant}, we have the
following result.

\begin{cor}If a connected graph is a planar 1-cycle resonant graph,
then it contains no even subdivision of $K_{2,3}$.
\end{cor}

%\begin{Remark} \label{relation}
%The converse of the last corollary is not valid. Since no cycle
%exists with edges in different blocks of a graph, we have that a
%graph $G$ contains no even subdivision of $K_{2,3}$ if and only if
%each 2-connected block of it contains no even subdivision of
%$K_{2,3}$. So if a bipartite graph contains no even subdivision of
%$K_{2,3}$, then by Theorem \ref{is 1-cycle resonant} each
%2-connected block of it is 1-cycle resonant. But $G$ may not be
%1-cycle resonant. This because the subgraph induced by the vertices
%not in any 2-connected block of $G$ may not have a perfect matching.
%\end{Remark}

We have seen in the previous theorem that 2-connected bipartite
graphs containing no even subdivision of $K_{2,3}$ are equivalent to
planar 1-cycle resonant graphs. In fact, various characterizations,
the construction and recognition algorithms for planar 1-cycle
resonant graphs have already been obtained  in \cite{Guo2, Guo1, Guo,
Xu}.  Before stating these results, we need to give some terminology
and notations.

Let $H$ be a subgraph of a connected graph $G$.  A \textit{bridge}
$B$ of $H$ is either an edge in $G-E(H)$ with two endvertices in
$H$, or a subgraph of $G$ induced by all edges incident with a
vertex in a  component $B^{\prime}$ of $G-V(H)$. An
\textit{attachment vertex} of a bridge $B$ to $H$ is a vertex in $H$
which is incident with an edge in $B$. Two bridges of a cycle $C$
avoid one another if all the attachment vertices of one bridge lie
between two consecutive  attachment vertices of the other bridge
along  $C$.

 Following Theorem \ref{is 1-cycle resonant} and Theorem 1
of Ref. \cite{Guo1}, we have the following characterizations.

\begin{thm} \label{5 equivalence}
 Let $G$ be a 2-connected bipartite planar graph. Then the following statements are equivalent:

 (1) $G$ contains no even subdivision of $K_{2,3}$,

 (2) $G$ is planar 1-cycle resonant,

 (3) For any cycle $C$ in $G$, $G-V(C)$ has no odd component,

 (4) For any cycle $C$ in $G$, any bridge of $C$ has exactly two attachment vertices which have different colors,

 (5) For any cycle $C$ in $G$, any two bridges of $C$ avoid one another. Moreover,
for any 2-connected subgraph $B$ of $G$ with exactly two attachment vertices, the attachment vertices of $B$ have different colors.
\end{thm}

The next result gives a structural description of planar 1-cycle
resonant graphs.
%If a graph $G$ has a bipartite ear decomposition satisfying the conditions
%in the following theorem, we say that $G$ can be obtained by
%\textit{1-CR-operations}.

\begin{thm} \cite{Guo} \label{ear}
A  2-connected graph $G$ is  planar 1-cycle resonant graph if and only if $G$ has a bipartite ear decomposition $G:=C_{0}+P_{1}
+\cdots +P_{k}$ such that $C_{0}$ is a cycle and each $P_{i}$ ($1\leq i\leq k$) satisfies that
(1) the endvertices $x$, $y$ of $P_{i}$ have different colors in $G_{i-1}=C_{0}+P_{1}
+\cdots +P_{i-1}$, (2) either $x$ and $y$ are adjacent in $G_{i-1}$ or $\{x, y\}$ is a vertex cut of $G_{i-1}$.
\end{thm}

This theorem can be  used  to construct bipartite graphs containing
no even subdivision of $K_{2,3}$. In addition, it  derives an
algorithm of time complexity $O(n^{2})$, where $n$ is the number of
vertices of $G$, to determine whether a 2-connected plane graph is
1-cycle resonant; See  \cite{Guo} for more details.

In \cite{Xu} a linear-time algorithm with respect to the number of
vertices to decide whether a 2-connected plane bipartite graph $G$
is 1-cycle resonant was provided. This algorithm is designed by
testing whether the attachment vertices of any bridge of the outer
cycle $C$ of $G$ satisfy statement (4) in Theorem \ref{5
equivalence} and the attachment vertices $u$, $v$ of any maximal
2-connected subgraph $H$ of any bridge $B$ of $C$ have different
colors. If the above conditions hold, we proceed recursively for
$G:=H$. Otherwise, $G$ is not 1-cycle resonant.

If a given planar bipartite graph $G$ is connected, we can implement
the above method to each 2-connected block of $G$ to test whether
it is 1-cycle resonant. If answers are all ``yes", the graph $G$
contains no even subdivision of $K_{2,3}$. Hence we can present a
linear time algorithm to determine whether a given planar bipartite
graph $G$ contains no even subdivision of $K_{2,3}$ in this
approach.

By the way, we turn to outerplanar graphs.  A graph is
\textit{outerplanar} if it has an embedding into the plane with
every vertex on the boundary of the exterior face. The following
characterizations for outerplanar graphs were given.

\begin{thm}\cite{Ch} \label{outerplanar}
A graph is outerplanar if and only if it contains no subdivision of
$K_{2,3}$ or $K_{4}$.
\end{thm}

\begin{thm}\cite{Sys} \label{no K4}
A  graph without triangles is outerplanar if and only if it contains no
subdivision of $K_{2,3}$.
\end{thm}

%\begin{Remark}\label{rem2}
 %  Yan and Zhang \cite{W.} constructed  Pfaffian orientations $G^{e}$ for outerplanar bipartite graphs $G$ and proved that
 % each cycle in $G^{e}$ is oddly oriented to obtain
 %$\pi(G,x)=\mathrm{det}(xI-A(G^{e}))$.
 %Using Theorem  \ref{no K4} and Corollary \ref{criterion}, we can obtain this result by a different approach.
%\end{Remark}

For bipartite graphs, we have the following characterization by
Theorem \ref{no K4}.

\begin{cor}\label{bipartite-out}A bipartite graph is outerplanar if and only if it
contains no subdivision of $K_{2,3}$.
\end{cor}

 From Corollary \ref{bipartite-out} and the
characterizations of a planar 1-cycle resonant graph, we can obtain
the following result immediately.

\begin{cor}
  Let $G:=C+ P_{1}+ \cdots+ P_{k}$ be a bipartite  ear
decomposition  obtained by Theorem \ref{ear}. If every ear $P_{i}$
is either a path of length 1 or joins two adjacent vertices of
$C\cup P_{1}\cup \cdots\cup P_{i-1}$ for $1\leq i\leq k$, then $G$
is outerplanar. If not,  $G$ contains a subdivision, but not an even
subdivision, of $K_{2,3}$.
\end{cor}

%%%%%%%%%%%%%%%%%%%%%%%%%%%%%%%%%%%%%%%%%%%%%%%%%%%%%%%%%%%%%%%%%%%%%%%%%%%%%%%%%%%%%%%%%%%%%%%%%%%%%%%%%%
%第六部分
\section{Permanental polynomials of some graphs}

In the last section some characterizations and recognition of
bipartite graphs containing no even subdivision of $K_{2,3}$ are
given. We now compute the permanental polynomials of such graphs. An
algorithm is presented firstly to construct their orientations with
each cycle being oddly oriented.
\subsection{An orientation algorithm}
The following algorithm has already  been provided in \cite{Lo} to
construct Pfaffian orientations for plane graphs. Here we will show
that, for a 2-connected bipartite graph $G$ containing no even
subdivision of $K_{2,3}$,  each cycle of $G$ is oddly oriented in
any  Pfaffian orientation of it.

\begin{Algorithm}\cite{Lo} \label{algorithm}
Let $G$ be a connected plane graph.\\
1. Find a spanning tree $T$ in $G$ and orient it arbitrarily.\\
2. Let $G_{1}=T$.\\
3. If $G_{i}=G$, stop. Otherwise, take an edge $e_{i}$ of $G$ not in $G_{i}$ such that $e_{i}$ and $G_{i}$ bound
   an interior face $f_{i}$ of $G$, and orient $e_{i}$ such that an odd number of edges on the boundary of $f_{i}$ is oriented clockwise.\\
4. Set $i+1=i$ and $G_{i+1}=G_{i}\cup \{e_{i}\}$. Go to step 3.
\end{Algorithm}

\begin{thm}\label{is Pfaffian} \cite{Lo}
 Let $G$ be a connected plane graph. The orientation $G^{e}$ given by Algorithm \ref{algorithm} is a Pfaffian orientation of $G$. Such an orientation can be constructed in polynomial time.
\end{thm}

\begin{thm} \label{orient thm}
     Let $G$ be a 2-connected  bipartite graph containing no even subdivision of $K_{2,3}$. Then each cycle is oddly oriented in
a Pfaffian orientation $G^{e}$ of it.
\end{thm}

\begin{proof}

Since $G^{e}$ is a Pfaffian orientation of $G$,
 each nice cycle is oddly oriented.  From Theorem \ref{is 1-cycle resonant},  each cycle of $G$ is  nice. So
each cycle of $G$ is oddly oriented in $G^{e}$.
\end{proof}

%\begin{Remark}
For a connected bipartite graph containing no even subdivision of
$K_{2,3}$, we can give an orientation so that  each cycle  is oddly
oriented: For each 2-connected block of $G$, we implement Algorithm
\ref{algorithm} to get its Pfaffian orientation; For each cut edge
of $G$ we  orient it arbitrarily. By Theorem \ref{orient thm} we can
see that for such an orientation of $G$ each cycle is indeed oddly
oriented.

\begin{figure}[htbp]
     \begin{center}
          \includegraphics [totalheight=4cm]{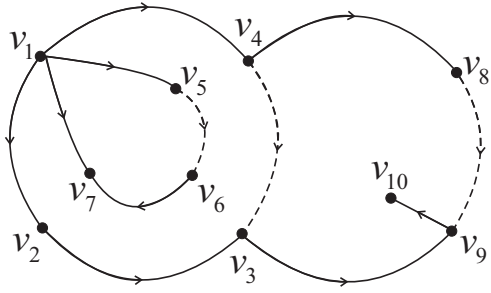}
          \caption{The graph with evenly oriented cycles.}\label{Example}
     \end{center}
\end{figure}

\begin{Remark} Implementing Algorithm \ref{algorithm} directly  to a connected plane
 bipartite graph containing no even subdivision of $K_{2,3}$, we obtain a  Pfaffian orientation, but cannot necessarily
 obtain the required orientation: an evenly oriented cycle may exist.
 See the graph $G$ in Figure \ref{Example} for example. The subgraph induced by
 the solid edges is a spanning tree $T$  with an orientation shown in Figure
 \ref{Example}. By Algorithm \ref{algorithm}, we first add the edge $v_{5}v_{6}$
 to $T$ and orient it from $v_{5}$ to $v_{6}$ so that the interior face bounded
 by it is oddly oriented.
 Then we add the edge $v_{3}v_{4}$ and  orient it
  from $v_{4}$ to $v_{3}$ so that the face boundary $v_1v_4v_3v_2v_1v_7v_6v_5v_1$ has an odd
  number of edges oriented clockwise along it. Finally  we add the edge $v_{8}v_{9}$ and  orient it
  from $v_{8}$ to $v_{9}$ so that the face boundary $v_4v_8v_9v_{10}v_9v_3v_4$ has an odd
  number of edges oriented clockwise along it. We can see that in the resulting Pfaffian orientation, the cycles
 $v_{1}v_{2}v_{3}v_{4}v_{1}$ and $v_{3}v_{4}v_{8}v_{9}v_{3}$ are evenly oriented.
\end{Remark}

\subsection{A computational approach}
Let  $G=(U, V)$ be a bipartite graph  with $|U|=|V|$. By choosing
suitable ordering of vertices, the skew adjacency matrix $A(G^{e})$
of an oriented graph $G^{e}$ has the form
$$A(G^e)=\begin{pmatrix}0  &  B \\
                    -B^{T}  &  0
\end{pmatrix},$$ where  $B$ is called the \textit{skew biadjacency matrix} of $G^{e}$. Let  $A$, $B$, $C$ and $D$ be $n\times n$
matrices with  $\mathrm{det}A\neq0$ and $AC=CA$. It is well-known that   $\mathrm{det}\begin{pmatrix}A  &  B \\
                    C  &  D
\end{pmatrix}$=
$\mathrm{det}\begin{pmatrix}AD-CB
\end{pmatrix}.$
Following from
$(xI)B^{T}=B^{T}(xI)$,
we obtain that $\mathrm{det}(xI-A(G^{e}))$ = $\mathrm{det}\begin{pmatrix}xI-\begin{pmatrix}0  &  B \\
                    -B^{T}  &  0
\end{pmatrix}\end{pmatrix}$ = $\mathrm{det}\begin{pmatrix}xI  &  -B \\
                    B^{T}  &  xI
\end{pmatrix}$ = $\mathrm{det}(x^{2}I+B^{T}B)$.  Using the result of Corollary \ref{criterion}
and Theorem \ref{orient thm}, we have the
following consequence.

\begin{thm}\label{4.4}
Let $G=(U, V)$ be a 2-connected bipartite graph  with $|U|=|V|$
containing no even subdivision of $K_{2,3}$. Then
$\pi(G,x)=\mathrm{det}(x^{2}I+B^{T}B)$, where $B$ is the skew
biadjacency matrix of a Pfaffian orientation $G^{e}$ of $G$.
\end{thm}

\subsection{Examples}
We now  give some examples to compute the permanental polynomials of
 bipartite graphs containing no even subdivision of $K_{2,3}$.
\vskip 2mm
 \noindent{\bf Example 1.}  Let  $G_{1}^{s}$ be the graph with $s$
pairwise internally disjoint paths of length three joining two given
vertices. See Figure \ref{Figure8a}.  We can see that $G_{1}^{s}$ is
a bipartite graph with $|U|$=$|V|$. By Theorems \ref{is 1-cycle
resonant} and \ref{ear} we know that $G_{1}^{s}$ contains no even
subdivision of $K_{2,3}$.

\begin{figure}[htbp]
         \centering
           \subfigure[]{\label{Figure8a}
           \includegraphics[totalheight=4cm]{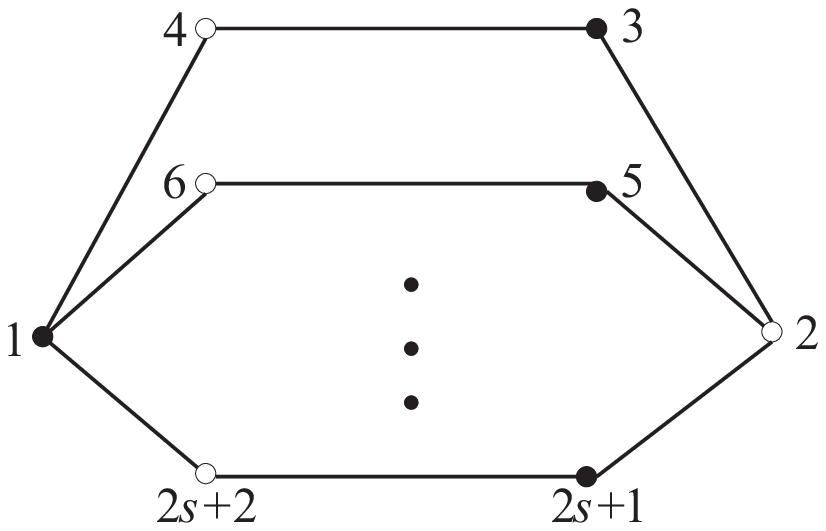}}\hspace{0.5cm}
           \subfigure[]{\label{Figure8b}
           \includegraphics[totalheight=4cm]{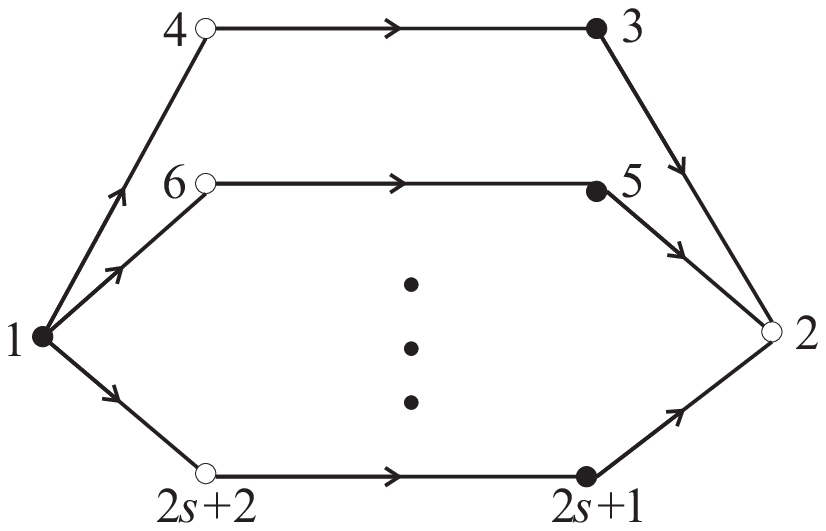}}\hspace{0.5cm}

           \caption{Graph $G_{1}^{s}$ and its Pfaffian orientation $(G_{1}^{s})^{e}$.}
\end{figure}

\begin{Lemma} \label{Dn}
$D_{n}=
        \mathrm{det}
        \begin{pmatrix}a_{1} & b & b & \cdots & b \\
               b & a_{2} & b & \cdots & b \\
               b & b & a_{3} & \cdots & b \\
              \vdots & \vdots & \vdots & \ddots & \vdots\\
               b & b & b & \cdots & a_{n}
        \end{pmatrix}$=
         $(1+b\sum_{i=1}^{n}\frac{1}{a_{i}-b})\prod_{i=1}^{n}(a_{i}-b)$,
         where $b\neq a_{i}, i=1, 2, \cdots, n$.
\end{Lemma}

\begin{proof}
$D_{n}=
\mathrm{det}
\begin{pmatrix}1 & b & b & b & \cdots & b \\
               0 & a_{1} & b & b & \cdots & b \\
               0 & b & a_{2} & b & \cdots & b \\
               0 & b & b & a_{3} & \cdots & b \\
               \vdots & \vdots & \vdots & \vdots & \ddots & \vdots\\
               0 & b & b & b & \cdots & a_{n}
\end{pmatrix}$=
$\mathrm{det}
\begin{pmatrix}1 & b & b & b & \cdots & b \\
               -1 & a_{1}-b & 0 & 0 & \cdots & 0 \\
               -1 & 0 & a_{2}-b & 0 & \cdots & 0 \\
               -1 & 0 & 0 & a_{3}-b & \cdots & 0 \\
               \vdots & \vdots & \vdots & \vdots & \ddots & \vdots\\
               -1 & 0 & 0 & 0 & \cdots & a_{n}-b
\end{pmatrix}$

=$\mathrm{det}
\begin{pmatrix}1+\frac{b}{a_{1}-b}+\cdots+\frac{b}{a_{n}-b} & b & b & b & \cdots & b \\
                & a_{1}-b &  &  &  &  \\
                &  & a_{2}-b &  &  & \mbox{\Huge 0} \\
                &  &  & a_{3}-b &  &  \\
  \mbox{\Huge 0} &  &  &  & \ddots & \\
                &  &  &  &  & a_{n}-b
\end{pmatrix}$

=$(1+b\sum_{i=1}^{n}\frac{1}{a_{i}-b})\prod_{i=1}^{n}(a_{i}-b)$.
\end{proof}

\begin{thm}
$\pi(G_{1}^{s},x)=(1+\frac{1}{x^{2}+s-1}+\frac{s}{x^{2}+1})(x^{2}+s-1)(x^{2}+1)^{s}$.
\end{thm}

\begin{proof} We choose a Pfaffian orientation $(G_{1}^{s})^{e}$ so that
each of the $s$ paths is oriented as a directed path from vertex 1 to
vertex 2 (see Figure \ref{Figure8b}). Since $G$ is 2-connected and contains
no even subdivision of $K_{2,3}$, by Theorem \ref{4.4} we have that
$\pi(G_{1}^{s},x)=\mathrm{det}(x^{2}I+B^{T}B)$, where
$B_{(s+1)\times(s+1)}$ the skew biadjacency matrix of
$(G_{1}^{s})^{e}$. As the labeling of vertices in Figure
\ref{Figure8b}, the skew biadjacency matrix $B$ has the following
form

$$B=\begin{pmatrix}0 & 1 & 1 & \cdots & 1 \\
                  1 &-1 & 0 & \cdots & 0 \\
                  1 & 0 &-1 & \cdots & 0 \\
                  \vdots & \vdots & \vdots & \ddots & \vdots\\
                  1 & 0 & 0 & \cdots & -1
\end{pmatrix},  \,\,\mbox{ and} \,\,
B^{T}B=\begin{pmatrix}s & -1 & -1 & \cdots & -1 \\
                      -1 & 2 & 1 & \cdots & 1 \\
                      -1 & 1 & 2 & \cdots & 1 \\
                      \vdots & \vdots & \vdots & \ddots & \vdots\\
                      -1 & 1 & 1 & \cdots & 2
\end{pmatrix}.$$
Hence
$\mathrm{det}(x^{2}I+B^{T}B)$\\
$=\mathrm{det}
\begin{pmatrix}x^{2}+s & -1 & -1 & \cdots & -1 \\
               -1 & x^{2}+2 & 1 & \cdots & 1 \\
               -1 & 1 & x^{2}+2 & \cdots & 1 \\
              \vdots & \vdots & \vdots & \ddots & \vdots\\
               -1 & 1 & 1 & \cdots & x^{2}+2
\end{pmatrix}$
=$\mathrm{det}\begin{pmatrix}x^{2}+s & 1 & 1 & \cdots & 1 \\
               1 & x^{2}+2 & 1 & \cdots & 1 \\
               1 & 1 & x^{2}+2 & \cdots & 1 \\
              \vdots & \vdots & \vdots & \ddots & \vdots\\
               1 & 1 & 1 & \cdots & x^{2}+2
\end{pmatrix}$\\
\\
$=(1+\frac{1}{x^{2}+s-1}+\frac{s}{x^{2}+1})(x^{2}+s-1)(x^{2}+1)^{s}$,
by Lemma \ref{Dn}.
\end{proof}

\noindent{\bf Example 2.} Let $G_{2}^{r}$=$(K_{1,r}\times
K_{2})^{*}$ be obtained from the Cartesian product $K_{1,r}\times
K_{2}$ by adding paths of length three joining the adjacent vertices
$u$ and $v$ of $K_{1,r}\times K_{2}$ with $u\in
V(K_{1,r}^{1})-\{x\}$ and $v\in V(K_{1,r}^{2})-\{x\}$ ($K_{1,r}^{1}$
and $K_{1,r}^{2}$ are the two copies of $K_{1,r}$ in $K_{1,r}\times
K_{2}$ and $x$ is the vertex of degree r in $K_{1,r}$). See Figure
\ref{Figure9a}. Similar to Example 1,  $G_{2}^{r}$ is a bipartite
graph with $|U|$=$|V|$ and contains no even subdivision of
$K_{2,3}$.

\begin{figure}[htbp]
         \centering
           \subfigure[]{\label{Figure9a}
           \includegraphics[totalheight=5cm]{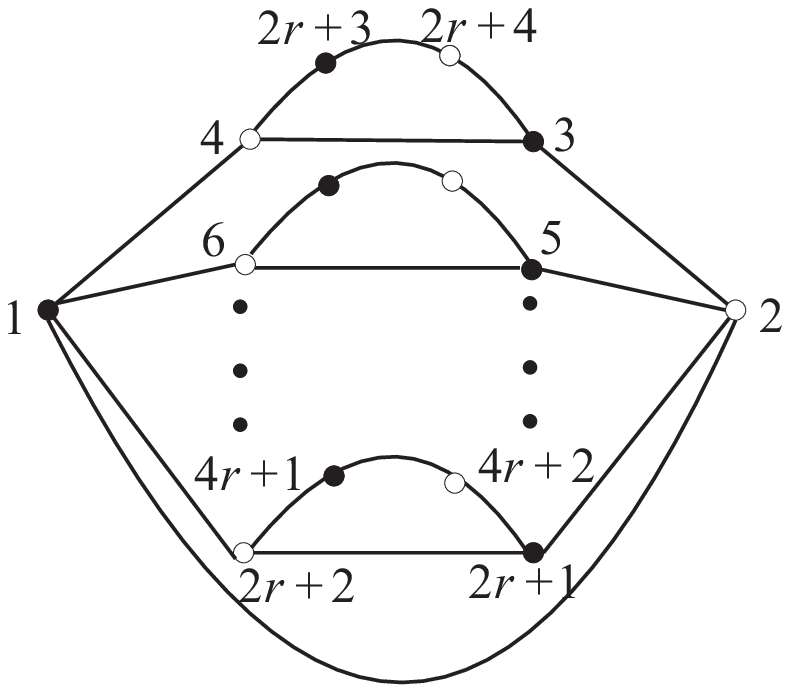}}\hspace{0.5cm}
           \subfigure[]{\label{Figure9b}
           \includegraphics[totalheight=5cm]{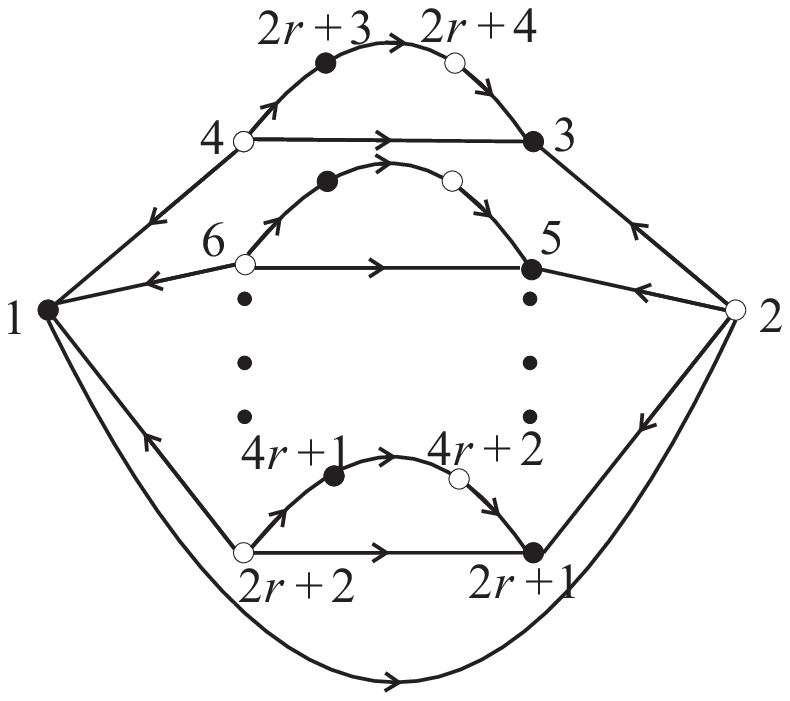}}\hspace{0.5cm}

           \caption{Graph $G_{2}^{r}$ and its Pfaffian orientation $(G_{2}^{r})^{e}$.}
\end{figure}

\begin{thm}
         $\pi(G_{2}^{r},x)=(1+\frac{r}{x^{2}+2})(x^{2}+2)^{2r-1}[x^{4}+(3+r)x^{2}+r+2]$.
\end{thm}

\begin{proof} Orientate  $G_{2}^{r}$ as follows. Direct the edges in the star $K_{1,r}^{1}$
from the vertices of degree one to the vertex of degree $r$ and the
        edges in the star
        $K_{1,r}^{2}$ receive the reverse direction;  Edges joining $K_{1,r}^{1}$ and $K_{1,r}^{2}$ are directed
        from $K_{1,r}^{1}$ to $K_{1,r}^{2}$; Each path of length three added to $K_{1,r}\times K_{2}$ is oriented
        as a directed path from  $K_{1,r}^{1}$ to  $K_{1,r}^{2}$ (see Figure \ref{Figure9b}). It can be easily
        checked that each face cycle is oddly oriented. Hence $(G_{2}^{r})^e$ is a Pfaffian orientation.
        Since  $G_{2}^{r}$ is 2-connected and contains no even subdivision of $K_{2,3}$,
        from Theorems \ref{4.4} we obtain that
        $\pi(G_{2}^{r},x)=\mathrm{det}(x^{2}I+B^{T}B)$, where the $(2r+1)\times (2r+1)$ matrix $B$ is the
        skew biadjacency matrix of  $(G_{2}^{r})^{e}$.
        By labeling the vertices of $(G_{2}^{r})^{e}$ as shown in Figure \ref{Figure9b}, we obtain that

        $$B=\begin{pmatrix}
                  1 & -1 & -1 & \cdots & -1 & 0 & 0 & \cdots & 0\\
                  -1 & -1 & 0 & \cdots & 0 & -1 & 0 & \cdots & 0\\
                  -1 & 0 & -1 & \cdots & 0 & 0 &  -1 & \cdots & 0\\
\vdots & \vdots & \vdots & \ddots & \vdots & \vdots & \vdots & \ddots & \vdots\\
                  -1 & 0 & 0 & \cdots & -1 & 0 &  0 & \cdots & -1\\
                  0 & -1 & 0 & \cdots & 0 & 1 & 0 & \cdots & 0\\
                  0 & 0 & -1 & \cdots & 0 & 0 & 1 &  \cdots & 0\\
\vdots & \vdots & \vdots & \ddots & \vdots & \vdots & \vdots & \ddots & \vdots\\
                  0 & 0 & 0 & \cdots & -1 & 0 & 0 & \cdots & 1\\
       \end{pmatrix}$$   and
  $$B^{T}B=\begin{pmatrix}
                  r+1 & 0 & 0 & \cdots & 0 & 1 & 1 & \cdots & 1\\
                  0 & 3 & 1 & \cdots & 1 & 0 & 0 & \cdots & 0\\
                  0 & 1 & 3 & \cdots & 1 & 0 & 0 &\cdots  & 0\\
\vdots  & \vdots & \vdots & \ddots &  \vdots & \vdots & \vdots & \ddots & \vdots\\
                  0 & 1 & 1 & \cdots & 3 & 0 & 0 & \cdots & 0\\
                  1 & 0 & 0 & \cdots & 0 & 2 & 0 & \cdots & 0\\
                  1 & 0 & 0 & \cdots & 0 & 0 & 2 & \cdots & 0\\
\vdots & \vdots &  \vdots & \ddots & \vdots & \vdots & \vdots &\ddots & \vdots\\
                  1 & 0 & 0 & \cdots & 0 & 0 & 0 & \cdots & 2\\
       \end{pmatrix}.$$

By the properties of determinants, we have that
$$\mathrm{det}(x^{2}I+B^{T}B)
       =\mathrm{det}\begin{pmatrix}
                  x^{2}+(r+1) & 0 & 0 & \cdots & 0 & 1 & 1 & \cdots & 1\\
                  0 & x^{2}+3 & 1 & \cdots & 1 & 0 & 0 &  \cdots & 0\\
                  0 & 1 & x^{2}+3 & \cdots & 1 & 0 & 0 & \cdots  & 0\\
                  \vdots  & \vdots & \vdots & \ddots &  \vdots & \vdots & \vdots & & \vdots\\
                  0 & 1 & 1 & \cdots & x^{2}+3 & 0 & 0 & \cdots & 0\\
                  1 & 0 & 0 & \cdots & 0 & x^{2}+2 & 0 & \cdots & 0\\
                  1 & 0 & 0 & \cdots & 0 & 0 & x^{2}+2 & \cdots & 0\\
                  \vdots & \vdots &  \vdots &  & \vdots & \vdots & \vdots  & \ddots & \vdots\\
                  1 & 0 & 0 & \cdots & 0 & 0 & 0 & \cdots & x^{2}+2\\
       \end{pmatrix}$$\\
       $=\mathrm{det}\begin{pmatrix}
                  x^{2}+(r+1) & 0 & 0 & \cdots & 0 & 0 & 0 & \cdots & 1\\
                  0 & x^{2}+3 & 1 & \cdots & 1 & 0 & 0 &  \cdots & 0\\
                  0 & 1 & x^{2}+3 & \cdots & 1 & 0 & 0 & \cdots  & 0\\
                  \vdots  & \vdots & \vdots & \ddots & \vdots & \vdots & \vdots &  & \vdots\\
                  0 & 1 & 1 & \cdots & x^{2}+3 & 0 & 0 & \cdots & 0\\
                  0 & 0 & 0 & \cdots & 0 & 2(x^{2}+2) & x^{2}+2 & \cdots & -x^{2}-2\\
                  0 & 0 & 0 & \cdots & 0 & x^{2}+2 & 2(x^{2}+2) & \cdots & -x^{2}-2\\
                  \vdots & \vdots &  \vdots &  & \vdots & \vdots & \vdots &\ddots & \vdots\\
                  1 & 0 & 0 & \cdots & 0 & -x^{2}-2 & -x^{2}-2 & \cdots & x^{2}+2
       \end{pmatrix}$\medskip

       (Since $B$ and $B^{T}B$ are of order $2r+1$, the two matrices as above are also of
       order $2r+1$.) \medskip\\
      $=(x^{2}+r+1)\mathrm{det}
       \begin{pmatrix}
                   x^{2}+3 & 1 & \cdots & 1 \\
                   1 & x^{2}+3 & \cdots & 1 \\
                   \vdots & \vdots & \ddots & \vdots\\
                   1 & 1 & \cdots & x^{2}+3
       \end{pmatrix}
       \mathrm{det}
       \begin{pmatrix}
                  2(x^{2}+2) & x^{2}+2 & \cdots & -x^{2}-2\\
                  x^{2}+2 & 2(x^{2}+2) & \cdots & -x^{2}-2\\
                  \vdots & \vdots  &\ddots & \vdots\\
                  -x^{2}-2 & -x^{2}-2 & \cdots & x^{2}+2
       \end{pmatrix}$\medskip\\
       $-
       \mathrm{det}
       \begin{pmatrix}
                  x^{2}+3 & 1 & \cdots & 1 \\
                   1 & x^{2}+3 & \cdots & 1 \\
                   \vdots & \vdots & \ddots & \vdots\\
                   1 & 1 & \cdots & x^{2}+3
       \end{pmatrix}
       \mathrm{det}
       \begin{pmatrix}
                  2(x^{2}+2) & x^{2}+2 & \cdots & x^{2}+2\\
                  x^{2}+2 & 2(x^{2}+2) & \cdots & x^{2}+2\\
                  \vdots & \vdots  &\ddots & \vdots\\
                  x^{2}+2 & x^{2}+2 & \cdots & 2(x^{2}+2)
       \end{pmatrix}$\medskip

    \noindent (The first three matrices as above are of order $r$ and
     the last one is of order $r-1$;   Lemma \ref{Dn} is used repeatedly)\medskip\\
       $=(x^{2}+r+1)(1+\frac{r}{x^{2}+2})(x^{2}+2)^{r}
       \mathrm{det}
       \begin{pmatrix}
                  x^{2}+2 & 0 & \cdots & -x^{2}-2\\
                  0 & x^{2}+2 & \cdots & -x^{2}-2\\
                  \vdots & \vdots  &\ddots & \vdots\\
                  0 & 0 & \cdots & x^{2}+2
       \end{pmatrix}$
       $-r(1+\frac{r}{x^{2}+2})(x^{2}+2)^{2r-1}$\medskip\\
       $=(x^{2}+r+1)(1+\frac{r}{x^{2}+2})(x^{2}+2)^{2r}
       -r(1+\frac{r}{x^{2}+2})(x^{2}+2)^{2r-1}$\medskip\\
       $=(1+\frac{r}{x^{2}+2})(x^{2}+2)^{2r-1}[x^{4}+(3+r)x^{2}+r+2]$.
\end{proof}

\noindent{\bf Example 3.} We consider outerplanar bipartite graphs.
For example, see Figure \ref{Figure10a}. If all the polygons in this
graph are hexagons, then the resulting graph is a catacondensed
hexagonal system  (see Figure \ref{Figure10b}).

\begin{figure}[htbp]
         \centering
           \subfigure[]{\label{Figure10a}
           \includegraphics[totalheight=4cm]{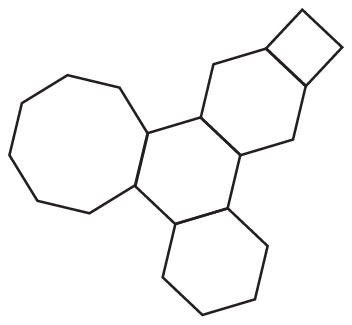}}\hspace{0.5cm}
           \subfigure[]{\label{Figure10b}
           \includegraphics[totalheight=4cm]{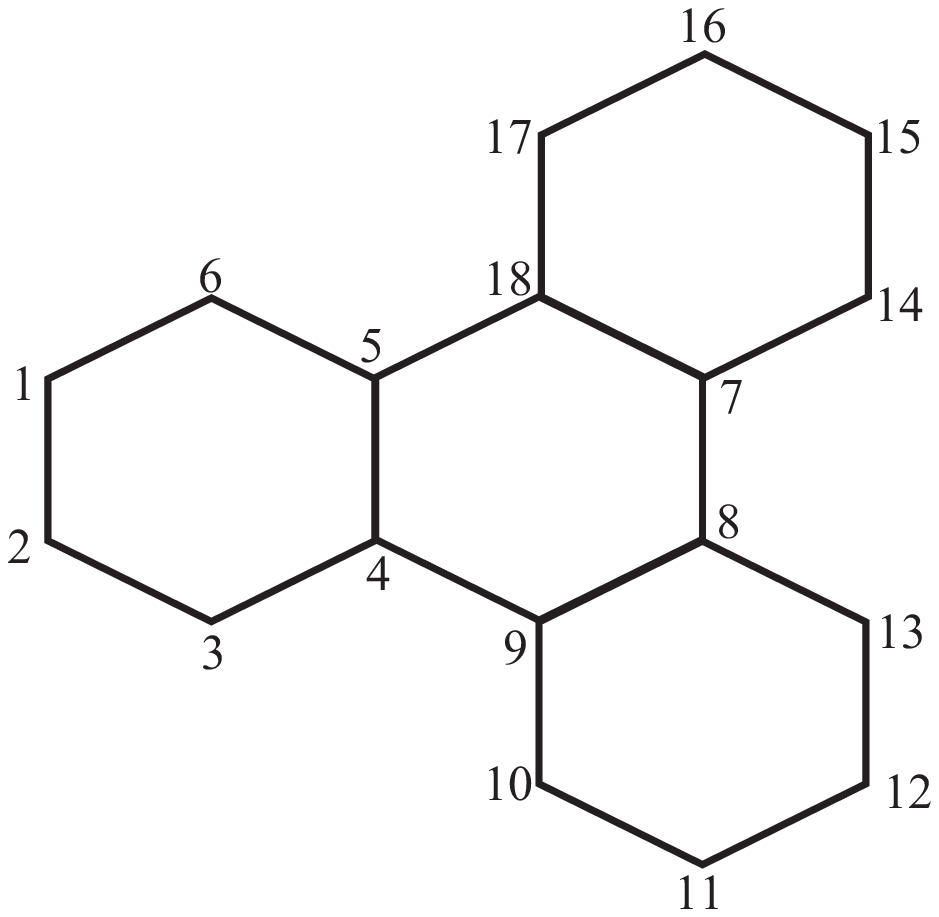}}\hspace{0.5cm}
           \subfigure[]{\label{Figure10c}
           \includegraphics[totalheight=4cm]{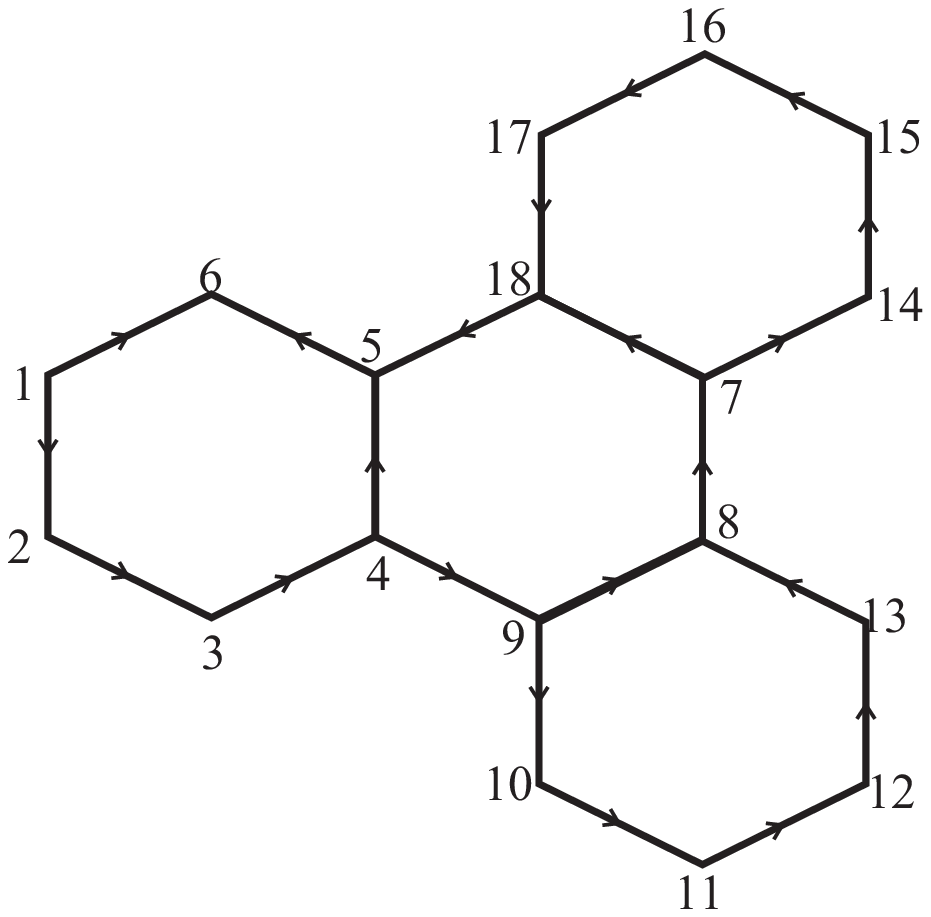}}\hspace{0.5cm}

           \caption{(a) An Outerplanar graph, (b) a catacondensed hexagonal system, and (c) an Pfaffian orientation.}
\end{figure}

Let $H$ be the graph in Figure \ref{Figure10b} with an orientation
$H^{e}$ in Figure \ref{Figure10c}  with each cycle oddly oriented.
Let $A(H^{e})$ be the skew adjacency matrix of $H^{e}$. We have that
$\pi(H,x)=\mathrm{det}(xI-A(H^{e}))$ by Theorem  \ref{no K4} and
Corollary \ref{criterion}. After computation we obtain that
$\pi(H,x)= 81+648x^{2}+2106x^{4}+3627x^{6}+
3645x^{8}+2223x^{10}+825x^{12}+180x^{14}+21x^{16}+x^{18}.$

 Borowiechi \cite{M.} ever proved that if a bipartite graph $G$ contains no cycle of length
$4s$, $s\in \{1, 2, \cdots\}$, and the characteristic polynomial
$\phi(G, x)=\Sigma_{k=0}^{[n/2]}(-1)^{k}a_{2k}x^{n-2k}$, then
$\pi(G,x)=\Sigma_{k=0}^{[n/2]}a_{2k}x^{n-2k}$. For example,
we can compute $\phi(H,
x)=-81+648x^{2}-2106x^{4}+3627x^{6}-3645x^{8}+2223x^{10}-825x^{12}
+180x^{14}-21x^{16}+x^{18}$. Hence the permanental polynomials of
the catacondensed hexagonal system can be computed in such two
methods, since they contains no cycle of length $4s$.

\begin{figure}[htbp]
\begin{center}
\subfigure[]{\label{Figure11a}
           \includegraphics[totalheight=4cm]{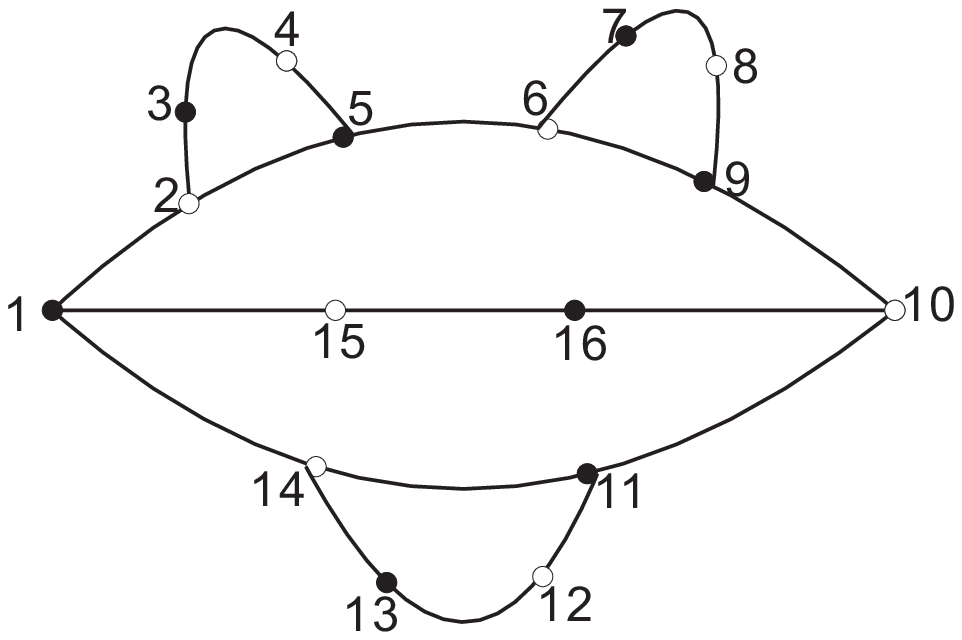}}\hspace{0.5cm}
           \subfigure[]{\label{Figure11b}
           \includegraphics[totalheight=4cm]{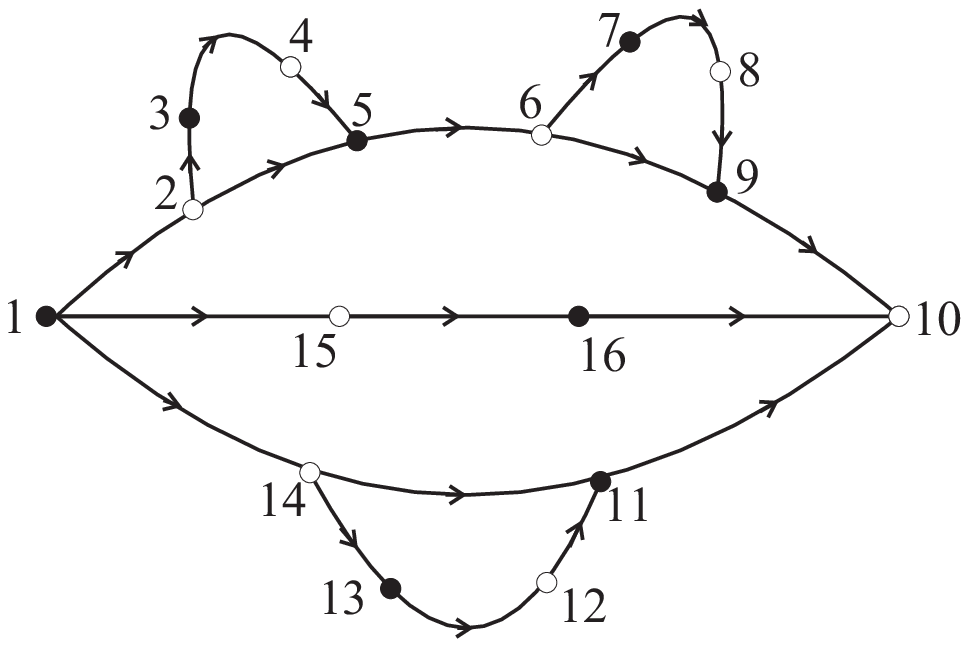}}\hspace{0.5cm}

           \caption{A graph containing no even subdivision of $K_{2,3}$ and its Pfaffian orientation.}
\end{center}
\end{figure}

\noindent{\bf Example 4.} Figure \ref{Figure11a} is a bipartite
graph $G$ containing no even subdivision of $K_{2,3}$. An
orientation $G^{e}$ obtained by Algorithm \ref{algorithm} is given
in Figure \ref{Figure11b}. Let $A(G^{e})$ be the skew adjacency
matrix of $G^{e}$. By Corollary \ref{criterion}, we have that
$\pi(G,x)=\mathrm{det}(xI-A(G^{e}))=196+1108x^{2}+2433x^{4}+2780x^{6}
+1832x^{8}+718x^{10}+ 164x^{12}+20x^{14}+x^{16}.$

%%%%%%%%%%%%%%%%%%%%%%%%%%%%%%%%%%%%%%%%%%%%%%%%%%%%%%%%%%%%%%%%%%%%%%%%%%%%%%%%%%%%%
%%References----------------------------------------------------


\begin{thebibliography}{6}

\bibitem{J.} J.A. Bondy and U.S.R. Murty, Graph Theory, Springer,
2008.
\bibitem{M.} M. Borowiecki, On spectrum and per-spectrum of graphs,
    Publ. Inst. Math., Nouv. S\'{e}r. 38(52) (1985) 31-33.
\bibitem{Pe} G.G. Cash, Permanental polynomials of smaller
                fullerenes, J. Chem. Inf. Comput. Sci. 40 (2000)
                1207-1209.
\bibitem{Ca} G.G. Cash, The permanental polynomial, J. Chem. Inf. Comput. Sci. 40
           (2000) 1203-1206.
\bibitem{Ca2}G.G. Cash,  A differential-operator approach to the permanental
polynomial, J. Chem. Inf. Comput. Sci. 42 (2002) 1132-1135.
\bibitem{Ch} G. Chartrand and F. Harary, Planar permutation graphs, Ann. Inst. Henri Poincar\`{e} B 3 (1967)
433-438.
\bibitem{B} D.M. Cvetkovi\'{c}, M. Doob and H. Sachs, Spectra of graphs, Academic Press, New York,
1980.
\bibitem{Fi}I. Fischer and C.H.C. Little, Even circuits of prescribed clockwise
      parity, Electron. J. Combin. 10 (2003) \#R45.
\bibitem{Guo2}X. Guo and F. Zhang, $k$-cycle resonance graphs, Discrete Math. 135 (1994) 113-120.
\bibitem{Guo1}X. Guo and F. Zhang, Planar $k$-cycle resonant graphs with $k$=1, 2, Discrete Appl. Math. 129 (2003) 383-397.
\bibitem{Guo}X. Guo and F. Zhang, Reducible chains of planar 1-cycle resonant graphs, Discrete Math. 275 (2004) 151-164.
\bibitem{I.} I. Gutman and G.G. Cash, Relations between the permanental and characteristic polynomials of
fullerenes and benzenoid hydrocarbons,
    MATCH Commum. Math. Comput. Chem. 45 (2002) 55-70.
\bibitem{Huo.} Y. Huo, H. Liang and F. Bai, An efficient algorithm for computing permanental polynomials of
graphs, Comput. Phys. Comm. 125 (2006) 196-203.
\bibitem{Kasum.}D. Kasum and N. Trinajsti\'{c}, I. Gutman, Chemical graph theory $\mathrm{III}$. On the permanental polynomial, Croat. Chem. Acta 54 (1981)
321-328.
\bibitem{Li} C.H.C. Little, A characterization of convertible (0,1)-matrices, J. Combin. Theory Ser. B 18 (1975) 187-208.
\bibitem{Lo}L. Lov\'{a}sz and M.D. Plummer, Matching Theory, Annals of Discrete Mathematics, Vol. 29, North-Holland, New York, 1986.
\bibitem{R.}R. Merris, K.R. Rebman and W. Watkins, Permanental
      polynomials of graphs, Linear Algebra Appl. 38 (1981) 273-288.
\bibitem{Sys} M. M. Sys{\l}o, Characterization of outplanar graphs, Discrete Math. 26 (1979) 47-53.
\bibitem{L.} L. Valliant, The complexity of computing the permanent, Theor. Comput. Sci. 8 (1979)
189-201.
\bibitem{Xu}Z. Xu and X. Guo, Construction and recognization of planar one-cycle resonant graphs, Graph Theory Notes, NY Acad. Sci.  XLII (2002) 44-48.
\bibitem{W.} W. Yan and F. Zhang, On the permanental polynomials of some graphs, J. Math. Chem. 35 (2004) 175-188.
\bibitem{Zhang} H. Zhang and F. Zhang, Plane elementary bipartite graphs, Discrete Appl. Math. 105 (2000) 291-311.


\end{thebibliography}
\end{document}